\documentclass{article}

\usepackage[affil-it]{authblk}
\usepackage{graphicx}
\usepackage{color,bbm,subfig,caption,bm,fontenc,currvita,url}

\usepackage{amsmath,amsthm}
\usepackage{amsfonts}
\usepackage{amssymb,color,makeidx,enumerate,bbm}

\makeatletter
\renewcommand{\pod}[1]{\allowbreak\mathchoice
  {\if@display \mkern 18mu\else \mkern 8mu\fi (#1)}
  {\if@display \mkern 18mu\else \mkern 8mu\fi (#1)}
  {\mkern4mu(#1)}
  {\mkern4mu(#1)}
}
\newtheorem{theorem}{Theorem}
\newtheorem{ther}{Theorem}[section]

\newtheorem{exa}{Example}[section]

\newtheorem{prop}{Proposition}[section]

\newtheorem{defi}{Definition}[section]
\newtheorem{lemma}[theorem]{Lemma}

\newtheorem{rem}{Remark}[section]

\newtheorem{cor}{Corollary}[section]

\newcommand{\argmin}{\arg\!\min}
\DeclareMathOperator{\spn}{span}

  \DeclareMathOperator*{\rank}{rank}
\def\e{{\bf e}}

\def\x{{\bf x}}

\def\y{{\bf y}}
\def\X{{\bf X}}
\def\Y{{\bf Y}}

\def\x{{\bf x}}

\def\y{{\bf y}}
\def\z{{\bf z}}
\def\s{{\bf s}}

\def\a{{\bf a}}

\newcommand{\msf}{\mathsf}

\def\F{{\bf F}}

\def\S{{\bf S}}

\def\v{{\bf v}}

\def\u{{\bf u}}
\def\f{{\bf f}}
\def\g{{\bf g}}
\def\h{{\bf h}}
\def\c{{\bf c}}

\def\bZ{\mathbb{Z}}

\def\bN{\mathbb{N}}
\def\bC{\mathbb{C}}
\def\C{\mathbb{C}}
\def\1{\mathbbm{1}}

\def\bR{\mathbb{R}}
\def\R{\mathbb{R}}
\def\bea{\begin{eqnarray}}
\def\eea{\end{eqnarray}}
\def\be{\begin{equation}}
\def\ee{\end{equation}}
\def\eps{\epsilon}

\def\rar{\rightarrow}
\usepackage{layout}

\begin{document}

\title{On Numerical Approximations of the Koopman Operator}

\date{\today}
\author{Igor Mezi\'c}
\affil{Mechanical Engineering and Mathematics, \\ University of California, \\
Santa Barbara, CA 93106}
\maketitle
\tableofcontents
\begin{abstract}
We study numerical approaches to computation of spectral properties of composition operators.
We provide a characterization of Koopman Modes in Banach spaces using Generalized Laplace Analysis. We cast the Dynamic Mode-Decomposition type methods in the context of Finite Section theory of infinite dimensional operators, and provide an example of a mixing map for which the finite section method fails. Under assumptions on the underlying dynamics, we provide the first result on the convergence rate under sample size increase in the finite-section approximation. We study the error in the Krylov subspace version of the finite section method and prove convergence in pseudospectral sense for operators with pure point spectrum. This result indicates that Krylov sequence-based approximations can have low error without an exponential-in-dimension increase in the number of functions needed for approximation.  \end{abstract}
\section{Introduction}

Spectral theory of dynamical systems shifts the focus of investigation of dynamical systems behavior away from trajectories in the state space 
and towards spectral features of an associated infinite-dimensional linear operator. Of particular interest is the  composition operator  - in measure-preserving setting called the Koopman operator - \cite{Koopman:1931,LasotaandMackey:1994,SinghandManhas:1993,MezicandBanaszuk:2004,Mezic:2005}, whose spectral triple - eigenvalues, eigenfunctions and eigenmodes - can be used in a variety of contexts, from model reduction \cite{Mezic:2005} to stability and control \cite{Mauroyetal:2020}. In practice,  we only have access to finite-dimensional data from observations or outputs of numerical simulations. Thus, it is important to study approximation properties of finite-dimensional numerical algorithms devised to compute spectral objects
\cite{hansen:2010}. Compactness is the property that imbues infinite-dimensional operators with quasi-finite-dimensional properties. Self-adjointness also helps in proving the approximation results. However, the composition operators under study here and rarely compact, or self-adjoint. In addition, in the classical, measure-preserving case the setting is that of unitary operators (and essentially self-adjoint generators for the continuous-time setting \cite{Tao}), but in the general, dissipative case, composition operators are neither. 

There are three main approaches to finding spectral objects of the Koopman operator.  The first, suggested already in \cite{MezicandBanaszuk:2000} is based on long time weighted averages over trajectories, rooted in ergodic theory of measure-preserving dynamical systems. An extension of that work that captures properties of continuous spectrum was presented in \cite{Kordaetal:2020}. This approach was named Generalized Laplace Analysis in \cite{Mezic:2013}, where concepts pertaining to dissipative systems were discussed also in terms of weighted averages along trajectories. In that sense, the ideas in this context provide an extension of ergodic theory for capturing transient (off-attractor) properties of systems.  For on-attractor evolution, the properties of the method acting on $L^2$ functions were studied in \cite{MezicandBanaszuk:2004,MauroyandMezic:2013}. The off-attractor case was pursued in \cite{MohrandMezic:2014} in Hardy-type spaces. This study was continued in \cite{Mezic:2019} to construct dynamics-adapted Hilbert spaces.  The advantage of the method is that it does not require the approximation of the operator itself, as it constructs eigenfunctions and eigenmodes directly from the data. It requires separate computation to first determine the spectrum of the operator, which is also done without constructing it. The second approach requires construction of an approximate operator acting on a finite-dimensional function subspace i.e. a finite section \cite{hansen:2010}. The best known such method is the Dynamic Mode Decomposition (DMD), invented in \cite{Schmid:2010} and connected to Koopman operator in \cite{Rowleyetal:2009}. The original DMD algorithm featured state observables.  The Extended Dynamic Mode Decomposition \cite{Williamsetal:2015} recognizes that nonlinear functions of state might be necessary to describe a finite-dimensional invariant subset of the Koopman operator and  provides an algorithm for finite-section approximation of the Koopman operator.  A study of convergence of such approximations is provided in \cite{kordaandmezic:2018}, but the convergence was established only along subsequences, and teh rate of convergence was not addressed. Here we provide the first result on the rate of convergence of the finite section approximation under assumptions on the nature of the underlying dynamics. It was observed already in \cite{MezicandBanaszuk:2000} that, instead of an arbitrary set of observables forming a basis, one can use observables generated by the dynamics  - namely time delays of a single observable filling a Krylov subspace - to study spectral properties of the Koopman operator. In the DMD context, the methods developed in this context are known under the name Hankel-DMD \cite{SusukiandMezic:2015,ArbabiandMezic:2017}. It is worth noticing that the Hankel matrix approach of \cite{SusukiandMezic:2015} is in fact based on the Prony approximation and requests a different sample structure than the Dynamic Mode Decomposition. The relationship between GLA and finite section methods was studied in \cite{MezicandArbabi:2017} The third approach is based on kernel integral operator combined with the Krylov subspace methodology \cite{dasandgiannakis:2019}, enabling approximation of continuous spectrum. While GLA and EDMD techniques have been extended to dissipative systems, the kernel integral operator technique is currently available for measure-preserving systems.

In this paper, we continue with the development of ergodic theory-rooted ideas to understanding and numerically computing the spectral triple for the Koopman operator. After some preliminaries, we start in section \ref{sect:GLA} with discussing properties of algorithms of Generalized Laplace Analysis type in Banach spaces. Such results have previously been obtained in Hardy-type spaces (\cite{MohrandMezic:2014}), and here we introduce a Gelfand-formula based technique that allows us to expand to general Banach spaces. We continue in section \ref{sect:FSM} with setting the finite-section approximation of the Koopman operator in the ergodic theory context. An explicit relationship of finite section coefficients to dual basis is established. Under assumptions on the underlying dynamics, we provide the first result on the convergence rate under sample size increase in the finite-section approximation. The error in the finite section approximation is analyzed. In section \ref{sec:kryl} we study finite section approximations of the Koopman operator based on Krylov sequences of time-delays of observables, and prove that under certain conditions, the approximation error decreases as the number of samples is increased, without dependence on the dimension of the problem. Namely, the Krylov subspace (Hankel-DMD) methodology has the advantage of convergence in the number of iterates and does not require a basis exponentially large in the number of dimensions. This solves the problem of the {\it choice of observables}, since the dynamics selects the basis by itself. In section \ref{sect:weakeig} we discuss an alternative point of view on the DMD approximations which is not related to finite sections, but samples of continuous functions on finite subsets of the state-space.  The concept of weak eigenfunctions is discussed, continuing the analysis in \cite{kordaandmezic:2018}. We conclude in section \ref{sect:concl}.


\section{Preliminaries}
For a dynamical system 
\be 
\dot \x=\F(\x),
\label{DSGen}
\ee
 defined on a state-space $M$ (i.e. $\x\in M$ - where we by slight abuse of notation identify a point in a manifold $M$ with its vector representation $\x$ in $\R^m$, $m$ being the dimension of the manifold), where $\x$ is a vector and $\F$ is a possibly nonlinear vector-valued smooth function, of the same dimension as its argument $\x$, denote by $\S^t(\x_0)$ the position at time $t$ of trajectory of   (\ref{DSGen}) 
 that starts at time $0$ at point $\x_0$. We  call the family of functions $\S^t$ the flow.

  Denote by $\g$ an arbitrary, vector-valued observable from $M$ to $\mathbb{R}^k$.  The value of this observable $\g$ that the system trajectory starting from $\x_0$ at time $0$ sees at time $t$ is 
 \be
 \g(t,\x_0)=\g(\S^t(\x_0)).
 \ee
Note that the space of all observables $\g$ is a linear vector space.  The family of operators $U^t,$ acting on the space of observables parametrized by time $t$ is defined by 
 \be
 U^t\g(\x_0)=\g(\S^t(\x_0)).
 \label{Koopdef}
 \ee
Thus, for a fixed time $\tau$, $U^\tau$ maps the vector-valued observable $\g(\x_0)$ to $\g(\tau,\x_0)$. We will call the family of operators $U^t$ indexed by time $t$ the Koopman operator of the continuous-time system (\ref{DSGen}).
This family was defined for the first time in \cite{Koopman:1931}, for Hamiltonian systems. In operator theory, such operators, when defined for general dynamical systems,  are often called composition operators\index{Composition operators}, since $U^t$ acts on observables by composing them with the mapping $\S^t$ \cite{SinghandManhas:1993}. 
Discretization of $\S^t$ for times $\tau,2\tau,...,n\tau,...$ leads to the $\tau$-mapping $T=S^\tau:M\rar M$ with the discrete dynamics
\be
\x'=T\x.
\ee
and the associated Koopman operator $U$ defined by
\be
Uf=f\circ T.
\ee
Let ${\cal F}$ be a space of observables and $U:{\cal F}\rar {\cal F}$ the Koopman operator associated with a map $T$ (note this means that $f\circ T\in {\cal F}$ if $f\in {\cal F}$). A function $\phi$ is an eigenfunction of $U$ associated with eigenvalue $\lambda$ provided 
\be
U\phi=\lambda \phi.
\ee
Let $\sigma(U)\in\C$ be the spectrum of $U$. The operator $U$ is called scalar \cite{dundford:1954} on ${\cal F}$ provided
\be
U=\int_{\sigma(U)}\beta dE,
\ee
where $E$ is a family of spectral projections forming resolution of teh identity, and spectral provided
\be
U=S+N,
\ee
where $S$ is scalar and $N$ quasi-nilpotent. Examples of functional spaces in which Koopman operators are scalar and spectral are given in \cite{Mezic:2019}. Let $\f\in {\cal F}$ be a vector of observables. For a scalar operator $U$ the Koopman mode of $\f$ associated with an eigenvalue $\lambda$ is given by 
\be
\s_\lambda=\f_\lambda./\phi,
\ee
where $./$ is component-wise division, $\phi$ is the unit norm eigenfunction associated with $\lambda$, and
\be
\f_\lambda=\f-\int_{\sigma(U)/\lambda}\beta dE(\f).
\ee
We assume that the dynamical system $T$ has a Milnor attractor ${\cal A}$ such that for every continuous function $g$,
for almost every $\x:M\rar M$ with respect to an a-priori measure $\nu$ on $M$ (without loss of generality as we can replace $M$ with the basin of attraction of  ${\cal A}$)
the limit
    \be
   g^*(\x)=\lim_{n\rar \infty} \frac{1}{n}\sum_{i=0}^{n-1} U^i g(\x),
    \ee
    exists. This is the case e.g. for smooth systems on subsets of $\R^n$ with Sinai-Bowen-Ruelle measures, where $\nu$ is the Lebesgue measure \cite{hunt1998unique}.

\section{Generalized Laplace Analysis}
\label{sect:GLA} \index{Generalized Laplace Analysis}
      Let $f(\x,\z)$ be a bounded field of observables $f(\x,\z):M\times A\rar \mathbb{R}^m$, \index{Field of observables} \index{Observable! field of} continuous in $\x$, where the observables are indexed over elements $\z$ of a set $A$, and $M$ is a compact metric space.
    We will occasionally drop the dependence on the state-space variable $\x$ and denote $f(\x,\z)=f(\z)$ and the iterates of $f$ by
    $f(T^i\x,\z)=f^i(\z)$. Let $U$ be the Koopman operator associated with a map $T:M\rar M$. We assume that $U$ is bounded, and acting on a Banach space of continuous functions $C$ (this does not have to be the space of {\it all} continuous functions on $M$, see the remark after the theorem).     \begin{ther}[Generalized Laplace Analysis]\label{GLA}
    
Let $\lambda_0,...,\lambda_K$ be simple  eigenvalues of $U$ such that $ |\lambda_0|\geq |\lambda_1|\geq...\geq |\lambda_K|,$ and there are no other points $\lambda$ in the spectrum of $U$ with $|\lambda| \geq |\lambda_K|$. Let $\phi_k$ be the eigenfunction of $U$ associated with $\lambda_k,k\in\{0,...,K\}$. Then, the Koopman mode associated with $\lambda_k$ is obtained by computing
\bea
f_k&=&\lim_{n\rar \infty} \frac{1}{n}\sum_{i=0}^{n-1} \lambda_k^{-i}\left(f(T^i \x,\z)-\sum_{j=0}^{k-1} \lambda_j^i\phi_j(\x)s_j(\z)\right) \nonumber \\
&=&\lim_{n\rar \infty} \frac{1}{n}\sum_{i=0}^{n-1} \lambda_k^{-i}\left(f^i(\z)-\sum_{j=0}^{k-1} \lambda_j^i f_j\right)
 \label{proj}
 \eea
where $f_k=\phi_k(\x) s_k(\z)$,  $\phi_k$ is an eigenfunction of $U$ with $|\phi_k|=1$ and $s_k$ is the k-th Koopman mode.
\end{ther}
\begin{proof} We introduce the operator 
\be
U_{\lambda_0}=\lambda_0^{-1}U.
\ee
Then, for some  function $g(\x),$ consider
\bea
U(\lim_{n\rar \infty} \frac{1}{n}\sum_{i=0}^{n-1} U_{\lambda_0}^i g(\x))&=&\lim_{n\rar \infty} \frac{1}{n}\sum_{i=0}^{n-1} \lambda_0^{-i}U^i g(T\x) \nonumber \\
&=& \lim_{n\rar \infty} \frac{1}{n}\sum_{i=0}^{n-1} \lambda_0^{-i}U^{i+1}g(\x)) \nonumber \\
&=& \lambda_0\lim_{n\rar \infty} \frac{1}{n}\sum_{i=0}^{n-1} \lambda_0^{-(i+1)}U^{i+1}g(\x) \nonumber \\
&=&  \lambda_0\left[\lim_{n\rar \infty} \frac{1}{n}\left(\sum_{i=0}^{n-1} U_{\lambda_0}^i g(\x)-g(\x)+\lambda_0^{-n}g(T^n\x)\right)\right] \nonumber \\
&=&  \lambda_0\left[\lim_{n\rar \infty} \frac{1}{n}\left(\sum_{i=0}^{n-1} U_{\lambda_0}^i g(\x)+\lambda_0^{-n}g(T^n\x)\right)\right],
\label{eq:gla}
\eea
where the last line is obtained by boundedness of $g$. Due to the boundedness of $U$ and continuity of $g$ we have 
\be 
\lim_{n\rar\infty}|\lambda_0^{-n}U^ng|\leq|g|.
\label{eq:Gel}
\ee
 This is obtained as the consequence of the so-called Gel'fand formula, that states that for a bounded operator $V$ on a Banach space $X$, $\lim_{n\rar\infty}|V^n|^{1/n}=\rho$ where $\rho$ is the spectral radius of $V$ \cite{Megginson:2012} (note that in our case $\rho=|\lambda_0|$). Thus, the last term in (\ref{eq:gla}) vanishes in the limit. Denoting
\be
g^*_{\lambda_0}(\x)=\lim_{n\rar \infty} \frac{1}{n}\sum_{i=0}^{n-1} U_{\lambda_0}^i g(\x),
\ee
where the convergence is again obtained from Gel'fand formula, utilizing the assumption on convergence of time averages and (\ref{eq:Gel}). Thus, we obtain 
\be
Ug^*_{\lambda_0}(\x)=\lambda_0 g^*_{\lambda_0}(\x)
\ee
and thus $g^*_{\lambda_0}(\x)$ is an eigenfunction of $U$ at eigenvalue $\lambda_0$. If we had a field of observables $f(\x,\z)$, parametrized by $\z$,
we get
\be
f^*_{\lambda_0}(\x,\z)=\phi_k(\x)s_j(\z),
\ee
since $f^*_{\lambda_0}(\x,\z)$ is an eigenfunction of $U$ at eigenvalue $\lambda_0$, so for every $\z$ it is just a constant (depending on $\z$) multiple 
of the eigenfunction $\phi_k(\x)$ of norm $1$. If we denote
\be
P_{\lambda_0}=\lim_{n\rar \infty} \frac{1}{n}\sum_{i=0}^{n-1} U_{\lambda_0}^i,
\ee
we can split the space of functions $C$ into the (possibly non-orthogonal) direct sum $P_{\lambda_0}C\bigoplus(I-P_{\lambda_0})C$.

Now, let $0<k<K$. Consider the space of observables 
\be
(I-P_{\lambda_0,\lambda_1,...,\lambda_{k-1}})C=(I-\sum_{j=0}^{k-1}P_{\lambda_j})C,
\ee
 complementary to the subspace $\Phi$ spanned by $\phi_j,0\leq j<k$. 
The operator $U|_\Phi$, the restriction of $U$ to $\Phi$ has eigenvalues $\lambda_0,...,\lambda_{k-1}$.  Since 
\be g_k=g-P_{\lambda_0,\lambda_1,...,\lambda_{k-1}}g
\ee does not have a component in $\Phi$, we can reduce the space of observables to $(I-P_{\lambda_0,\lambda_1,...,\lambda_{k-1}})C$, on which $U_{\lambda_k}$ satisfies the assumptions of the theorem, 
and obtain
\be
U(g_k)^*_{\lambda_k}(\x)=\lambda_k (g_k)^*_{\lambda_k}(\x).
\ee
If we have a field of observables  $f(\x,\z)$, then
\be
f_k(\x,\z)=f(\x,\z)-P_{\lambda_0,\lambda_1,...,\lambda_{k-1}}f,
\ee
and thus
\be
f_k(\x,\z)=\phi_k(\x)s_k(\z).
\ee
\end{proof}
In other words, $f_k$ is the skew-projection of the field of observables $f(\x,\z)$ on the eigenspace of
the Koopman operator associated with the eigenvalue $\lambda_k$.   
\begin{rem} The assumptions on eigenvalues in the above theorem would not be satisfied for dynamical systems whose eigenvalues are dense on the unit circle (e.g. a map that, as $n\rar\infty$ approaches a unit circle in the complex plane on which the dynamics is given by $z'=e^{i\omega}z$, where $\omega$ is irrational w.r.t. $\pi$). However, in such case the space of functions can be restricted to the span of functions $e^{ik\theta}, k=1,..,N, \theta \in [0,2\pi)$, and the requirements of the theorem would be satisfied. This amounts to restricting the observables to a set with finite resolution, which is standard in data analysis.
\end{rem}

In principle, one can find the full spectrum of the Koopman operator by performing Generalized Laplace Analysis, where Theorem \ref{GLA} is used on some function $g(\x)$ starting from 
the unit circle, successively subtracting parts of the signal corresponding to eigenvalues with decreasing $|\lambda|$. In practice, such computation can be unstable, since at large $t$ it involves a multiplication of very large with a very small number. In addition, the eigenvalues are typically not known a-priori. In the next section we describe the finite section method, in which the operator is represented in a basis, and a finite-dimensional truncation of the resulting infinite matrix - a finite section - is used to approximate its spectral properties. Under some conditions \cite{kordaandmezic:2018}, increasing the dimension of the finite section and the number of sample points, eigenvalues of the operator can be obtained.
\section{The Finite Section Method }\label{sect:FSM}
The GLA method for approximating eigenfunctions (and thus modes) of the Koopman operator, analyzed in the previous section, was proposed initially in \cite{MezicandBanaszuk:2000,MezicandBanaszuk:2004,Mezic:2005} in the context of on-attractor (measure-preserving) dynamics, and extended to off-attractor dynamics in \cite{Mezic:2013,MohrandMezic:2014,mauroy2013isostables,ArbabiandMezic:2017}. It is predicated on the knowledge of (approximate) eigenvalues, but it does not require the knowledge of an approximation to the Koopman operator. In contrast, DMD-type methods, invented initially in \cite{Schmid:2010} without the Koopman operator background, and connected to the Koopman operator setting in \cite{Rowleyetal:2009} produce a matrix approximation to the Koopman operator. There are many forms of the DMD methodology, but all of them require a choice of a finite set of observables that span a subspace. In this section we analyze such methods in the context of finite section of the operator and explore connections to the dual basis.
\subsection{Finite Section and the Dual Basis}
Consider the Koopman operator acting on an observable space ${\mathcal F}$ of functions on the state space $M$, equipped with the complex inner product $\left<\cdot,\cdot\right>$,\footnote{Note hat we are using the complex inner product linear in the first argument here. The physics literature typically employs the so-called Dirac notation, where the inner product is linear in its second argument.} and let
$\{f_j\},\ j\in \bN$ be an orthonormal basis on ${\mathcal F}$, such that, for any function $f\in {\mathcal F}$ we have
\be
f=\sum_{j\in \bN}c_jf_j.
\ee
Let
\be
u_{kj}=\left<Uf_j,f_k\right>.
\label{eq:inn}
\ee
Then,
\be
(Uf)_k=\left<Uf,f_k\right>=\sum_{j\in \bN}c_j\left<Uf_j,f_k\right>=\sum_{j\in \bN} u_{kj} c_j.
\ee

The basis functions do not necessarily need to be orthogonal. Consider the action of $U$ on an individual, basis function $f_j$:
\be
Uf_j=\sum_{k\in \bN}u_{kj} f_k,
\ee
where $u_{kj}$ are now just coefficients of $Uf_j$ in the basis.
We obtain
\be
Uf=\sum_{j\in \bN}c_jUf_j=\sum_{j\in \bN}c_j\sum_{k\in \bN}u_{kj} f_k=\sum_{k\in \bN}\left(\sum_{j\in \bN}u_{kj}c_j\right) f_k,
\ee
and we again have 
\be
(Uf)_k=\sum_{j\in \bN} u_{kj}c_j.
\ee
As in the previous section, associated with any linear subspace $\mathcal G$ of $\mathcal F$, there is a projection onto
it, denoted $P = P^2$, that we can think 
of as projection ``along" the space $(I-P){\mathcal F}$, since, for any $f\in {\mathcal F}$ we have
\be P(I-P)f=(P-P^2)f=0,
\ee
and thus any element of $(I-P){\mathcal F}$ has projection $0$. 
We denote by $\tilde U$ the infinite-dimensional matrix with elements $u_{kj},\ k,j \in \bN$. Thus, the finite-dimensional section of the matrix
\be
\tilde U_{n}=\begin{bmatrix}
		 u_{11} &  u_{12} & \cdots &   &  u_{1n}\\
		  u_{21}&  u_{22} &        &   &  u_{2n}\\
		&  &        &  & \\
		\vdots & & \ddots & &\vdots\\
		 u_{n1} & u_{n2} & \cdots & & u_{nn}
	\end{bmatrix},
	\label{comp}
\ee
is the so-called compression of $\tilde U$ that satisfies
\be
\tilde U_n=P_n\tilde U P_n,
\ee
where $P_n$ is the projection ``along" $
(I-P_n){\mathcal F}$ to the span of the first $n$ basis functions, $\spn({f_1,...,f_n})$. 

The key question now is: how are the eigenvalues of $\tilde U_n$ related to the spectrum of the infinite-dimensional operator $U$? This was first addressed in \cite{kordaandmezic:2018}
\begin{exa}
Consider the translation $T$ on the circle $S^1$ given by
\be
 z'=e^{i\omega}z, \ z\in S^1,
\ee
Let $f_j=e^{ij\theta},\ \theta\in [0,2\pi).$
Then, 
\be
Uf_j=f_j\circ T=e^{ij\omega}e^{ij\theta}.
\ee
Thus, from (\ref{eq:inn}) $u_{kj}=\delta_{kj}e^{ij\omega}$, where $\delta_{kj}=1$ for $k=j$ and zero otherwise (the Kronecker delta), and $\tilde U$ is a diagonal matrix. In this case, the finite section method  provides us with the subset of the exact eigenvalues of the Koopman operator.
\end{exa}
The following example shows how careful we need to be with the finite-section method when the underlying dynamical system has chaotic behavior:
\begin{exa} \label{exa:mix}
Consider the map $T$ on the circle $S^1$ given by
\be
 z'=z^2, \ z\in S^1,
\ee
This is a mixing map that does not have any eigenvalues of the Koopman operator on $L^2(S^1)$ \cite{ArnoldandAvez:1968}.
Let $f_j=e^{ij\theta},\ \theta\in [0,2\pi).$
Then, 
\be
Uf_j=f_j\circ T=e^{ij2\theta}.
\ee
Let 
\be f(\theta)=\sum_{j \in \bZ}c_je^{ij\theta}.
\ee
Then
\be Uf(\theta)=\sum_{j \in \bZ}c_je^{i2j\theta}.
\ee
Thus,  $\tilde U_n$ is given by
\be
\tilde U_{n}=\begin{bmatrix}
		 0 &  0& \cdots &   &  0\\
		  1&  0 &        &   &  0\\
		  0&  0 &        &   &  0\\
		  0&  1 &        &   &  0\\
		&  &        &  & \\
		\vdots & & \ddots & &\vdots\\
		 0 & 0 & \cdots & & 0
	\end{bmatrix},
	\label{comp}
\ee  
provided $n\neq 2k,k\in \bN$. In this case, the finite section method fails, as $U_{n}$ has eigenvalue $0$ of multiplicity $n$. This example illustrates 
how the condition in  \cite{kordaandmezic:2018} that the weak convergence of a subsequence of eigenfunctions of $\tilde U_N$ to a function $\phi$ must be accompanied by the requirement $||\phi||\neq 0$ in order 
that the limit of the associated subsequence of eigenvalues converges to a true eigenvalue of the Koopman operator. In particular, no subsequence of eigenvalues in this case converge to the true eigenvalue of the Koopman operator, since the map is measure preserving, and thus its eigenvalues are on the unit circle. The example shows the peril of applying the finite section method to find eigenvalues of the Koopman operator  when the underlying dynamical system has a part  continuous \cite{Mezic:2005} (in this case, Lebesgue \cite{ArnoldandAvez:1968}) spectrum. Continuous spectrum is effectively dealt with in \cite{Kordaetal:2020,govindarajanetal:2019} using harmonic analysis and periodic approximation methods, respectively.\end{exa}

\index{Finite section}
To apply the finite-section methodology of approximation of the Koopman operator, we need to estimate coefficients $u_{kj}$ from data.
If we have access to measurements of $N$ orthogonal functions $f_1,..,f_N$ on $m$ points on state space, as indicated in \cite{kordaandmezic:2018}, assuming ergodicity, this becomes possible:
\begin{ther}
\label{theror}
Let $\{f_1,..,f_N\}$ be an orthogonal set of functions in $L^2(M,\mu)$ and let $T$ be ergodic on $M$ with respect to an invariant measure $\mu$. Let $\x_l,l\in \bN$ be a trajectory on $M$. Then for almost any $\x_1\in M$
\be
u_{kj}=\lim_{m\rar\infty}\frac{1}{m}\sum_{l=1}^m f_k^c(\x_l) f_j \circ T(\x_l)=\lim_{m\rar\infty}\frac{1}{m}\sum_{l=1}^m f_k^c(\x_l) f_j (\x_{l+1})
\ee
\label{ther:finsecdata}
\end{ther}
\begin{proof}
This is a simple consequence of the Birkhoff ergodic theorem  (\cite{Petersen:1983}). Recall that 
\be
u_{kj}=\left<Uf_j,f_k\right>=\int_M Uf_jf_k ^cd\mu,
\ee
and the last expression is equal to
\be
\lim_{m\rar\infty}\frac{1}{m}\sum_{l=1}^m f_k^c(\x_l) f_j \circ T(\x_l),
\ee
by the Birkhoff Ergodic Theorem applied to the function $Uf_j f_k^c $.
\qed
\end{proof}

In the case of non-orthonormal basis, denote by $\hat f_k$ the dual basis vectors, such that
\be
\left<f_j,\hat f_k\right>=\delta_{jk},
\ee
where $\delta_{jj}=1$ for any $j$, and $ \delta_{jk}=0$ if $j\neq k$. 
For the infinite-dimensional Koopman matrix coefficients we get 
\be
u_{kj}=\left<Uf_j, \hat f_k\right>.
\ee
Let's consider the {\it finite} set of independent functions $\tilde \f=\{f_1,..,f_N\}$ and the associated dual set $\{\hat g_1,..,\hat g_N\}$ in the span  $\tilde {\mathcal F}$ of $\tilde \f$, that satisfy
\be
\left<f_j,\hat g_k\right>=\delta_{jk}.
\ee
Note that the functions $\hat g_k$ are unique, since they are each orthonormal to $n-1$ vectors in $\tilde {\mathcal F}$. 
Let 
\be
{\mathcal F}=\tilde {\mathcal F}+\tilde {\mathcal F}^T,
\ee
and $P_{\tilde {\mathcal F}}$ the orthogonal\footnote{This in effect assumes all the remaining basis functions are orthogonal to $\tilde {\mathcal F}$.}  projection on $\tilde {\mathcal F}$. Then
\be
\hat g_k=P_{\tilde {\mathcal F}} \hat f_k,
\ee
since, by self-adjointness of orthogonal projections, and $P_{\tilde {\mathcal F}} \hat f_k \in \tilde {\mathcal F}$
\be
\left<f_j, P_{\tilde {\mathcal F}} \hat f_k\right>=\left<P_{\tilde {\mathcal F}} f_j,  \hat f_k\right>=\left< f_j,  \hat f_k\right>=\delta_{jk}
\ee
Now, we have
\be
\label{eq:projnonor}
\tilde u_{kj}=\left<Uf_j,\hat g_k\right>=\left<Uf_j,P_{\tilde {\mathcal F}}\hat f_k\right>=\left<P_{\tilde {\mathcal F}}Uf_j,\hat f_k\right>
\ee
and thus, since $f_j\in {\mathcal F}$, the coefficients  $\tilde u_{kj}$ are the elements of the finite section $P_{\tilde {\mathcal F}}U P_{\tilde {\mathcal F}}$ in the basis $\tilde \f$.
 It is again possible to obtain $\tilde u_{kj}$ from data: 
\begin{ther}
\label{thernonor}
Let $\{f_1,..,f_N\}$ be a non-orthogonal set of functions in $L^2(M,\mu)$ and let $T$ be ergodic on $M$ with respect to an invariant measure $\mu$. Let $\x_l,l\in \bN$ be a trajectory on $M$. Then, for almost any $\x_1\in M$
\be
\label{eq:finsecentries}
\tilde u_{kj}=\lim_{m\rar\infty}\frac{1}{m}\sum_{l=1}^m f_j \circ T(\x_l) \hat g_k^c(\x_l)=\lim_{m\rar\infty}\frac{1}{m}\sum_{l=1}^m  f_j (\x_{l+1}) \hat g_k^c(\x_l),
\ee
where, for any finite $m$, $\hat g_k^c(\x_l),l=1,...,m$ are obtained as rows of the matrix $(F^{\dagger}F)^{-1}F^{\dagger}$, where
\be
F=\left[f_1(\X) \ f_2(\X)\ ... \ f_N(\X)\right],
\ee
$F^{\dagger}=(F^c)^{T}$ is the conjugate (Hermitian) transpose \index{Transpose! conjugate} \index{Conjugate transpose} \index{Hermitian transpose} of $F$
, and $f_j(\X)$ is the column vector $(f_j(\x_1) \ ... \ f_j(\x_m))^T$.
 \label{ther:finsecdata1}
\end{ther}
\begin{proof}
The fact that $\hat g_k(\x_l),l=1,...,m$ are obtained as rows of the matrix $(F^{\dagger}F)^{-1}F^{\dagger}$ follows from
\be
(F^{\dagger}F)^{-1}F^{\dagger}F=I_N
\ee
where $I_N$ is the $N\times N$ identity matrix. The rest of the proof is analogous to the proof of theorem \ref{ther:finsecdata}.

\end{proof}
\begin{rem} The key idea in the above results, Theorems \ref{theror}  and  \ref{thernonor} is that we sample the functions $f_i,i=1,...,N$ and the dual basis $g_k, k=1,...,N$ on $m$ points in the state space,
and then take the limit $m\rar\infty$. Thus besides approximating the action of $U$ using the finite section $\tilde U_N$, we also approximate individual functions $f_j,g_k$ by their sample on $m$ points. The corollary of the theorems is that the finite sample approximations $\tilde U_{N,m}$, obtained by setting the coefficients
\be
\tilde u_{kj,m}=\frac{1}{m}\sum_{l=1}^m f_j \circ T(\x_l) \hat g_k^c(\x_l)
\label{eq:finsam}
\ee
converges to $\tilde U_N$ as $m\rightarrow \infty$. This result has been obtained in \cite{klusetal:2016}, without the use of the dual basis, relying on the Moore-Penrose pseudoinverse, the connection with which we discuss next.
\end{rem}
We call $F$ the {\it data matrix}. \index{Data matrix} Note that the matrix $F^+=(F^{{\dagger}}F)^{-1}F^{{\dagger}}$ is the so-called Moore-Penrose pseudoinverse \index{Moore-Penrose pseudoinverse} of $F$. Using matrix notation, from (\ref{eq:finsecentries}) the approximation of the finite section can be written as
\be
\tilde U_N^a=F^+ F(T(\X))=F^+ F',
\label{eq:finsecmat}
\ee
where $\X=(\x_1,...,\x_m)^T$, 
\be
F'=F(T(\X))=\left[ \begin{matrix} f_1(T\X) \ f_2(T\X)\ ...\ f_{N}(T\X)\end{matrix}\right].
\ee
and $ f_k(T\X)$ is the column vector
\be\left[
\begin{matrix}
f_k(T\x_1)\\
\vdots \\
f_k(T\x_m)\\
\end{matrix}
\right].
\ee

If we now assume that there is an eigenfunction-eigenvalue pair $\lambda,\phi$ of $U$ such that $\phi\in \spn{\tilde{\mathcal F}}$ then 
\be
P_{\tilde {\mathcal F}}UP_{\tilde {\mathcal F}}\phi=P_{\tilde {\mathcal F}}U\phi=U\phi=\lambda\phi.
\ee
Thus, the eigenvalue $\lambda$ will be in the spectrum of $\tilde U_N$. More generally, it is known that an operator $U$ and a projection $P_{\tilde {\mathcal F}}$ commute if and only if ${\tilde {\mathcal F}}$ is an invariant subspace of $U$. Thus, the spectrum of the finite-section operator $\tilde U_N$ is a subset of the spectrum of $U$ for the case when ${\tilde {\mathcal F}}$ is an invariant subspace.

 If an eigenfunction $\phi$ of $U$  is in ${\tilde {\mathcal F}},$ it can be obtained from an eigenvector $\a$ of the finite section
$\tilde U_N$ as
\be
\phi=\a\cdot \tilde \f=\sum_{k=1}^N a_k f_k,
\label{eqn:eigen}
\ee
where $\a=(a_1,...,a_N)$ satisfies $\tilde U_N \a=\lambda \a$, since, for such $\phi$,
\be
U\phi=\a\cdot U \tilde \f=\lambda\phi=\lambda \a\cdot  \tilde \f=U_N\a \cdot  \tilde \f.
\label{eq:eigffs}
\ee
We have introduced above the {\it dot} notation, that produces a function in ${\mathcal F}$ from an $N$-vector $\a$ and a set of functions $\tilde \f$.
\begin{rem} \label{rem:genmeas}The theorems  \ref{ther:finsecdata} and  \ref{ther:finsecdata1} are convenient in their use of sampling along trajectory and an invariant measure, thus enabling construction of  finite section representations of the  Koopman operator from a single trajectory. However, the associated space of functions $L^2(\mu)$ is restricted since the resulting spectrum is on the unit circle. Choosing a more general measure $\nu$ that has support in the basin of attraction is possible. Namely, when we construct the finite section, we then use a sequence $\x_l, l=1,...,m$ of points that weakly converge to the measure $\nu$, and their images under $T$, $\y_l=T(\x_l)$. This is the approach in \cite{klusetal:2016}. The potential issue with this approach is the choice of space - typically, $L^2(\nu)$ will have a very large spectrum, for example filling the entire unit disk of the complex plane \cite{ridge:1973}. In contrast, Hilbert spaces adapted to the dynamics of a dissipative systems can be constructed \cite{Mezic:2019}, starting from the ideal of continuous functions that vanish on the attractor, enabling a natural setting for computation of spectral objects for dissipative systems.
\end{rem}
Koopman mode \index{Koopman mode} is the projection of a field of observables on an eigenfunction of $U$. Approximations of Koopman modes can also be obtained using a finite section. Let $\tilde U_N$ be a finite section of $\tilde U$. Let $\h:M\rar\C^K$ be a vector observable (thus a field of observables indexed over a discrete set). \index{Observable! field of} \index{Field of observables} Then the Koopman mode $s_\lambda(\h)$ associated with the eigenvalue $\lambda$ of $U$ is obtained as
\be
s_\lambda(\h)=\left<\h,\hat \phi\right>\phi,
\ee
where $\phi,\hat \phi$ are the eigenfunction and the dual eigenfunction associated with the eigenvalue $\lambda$.
Let $\a_j,j=1,...,N$ be  eigenvectors of $\tilde U_N$, and thus the associated eigenfunctions of the finite section are
\be
 \phi_j=\a_j \cdot \tilde \f,\ j=1,...,N
\ee
where $\a_j=(a_{j1},...,a_{jN})$. Then, we get the dual basis
\be
 \hat \phi_j=\left<\hat \a_j, \hat \g\right>,j=1,...,N
\ee
where
\be
\left<\hat \a_j,\a_k\right>=\delta_{jk}.
\ee
This is easily checked by expanding:
\bea
\left<\phi_j,\hat \phi_j \right>&=&\left<\sum_{k=1}^N a_{jk}f_k,\sum_{l=1}^N \hat a_{il}\hat g_l\right>  \nonumber \\
&=&\sum_{k=1}^N \sum_{l=1}^N a_{jk} \hat a_{il}^c \left<f_k,\hat g_l\right>\nonumber \\
&=&\sum_{k=1}^N  a_{jk} \hat a_{ik}^c=\delta_{ji}.
\eea

Thus, the approximation $\tilde s_j(\h)$ to the Koopman mode $s_j(\h)$ associated with the eigenvalue $\lambda_j$ of the finite section reads
\be
\tilde s_j(\h)=\left<\h,\hat \phi_j\right>\phi_j=\sum_{k=1}^N\hat a_{jk}\left<\h,\hat g_k\right>\phi_j.
\ee
Now assume that $\h=\tilde \f$, 
\be
\left<\tilde \f,\hat \phi_j\right>\phi_j=\sum_{k=1}^N(\hat\a_j)_k\left<\tilde\f,\hat g_k\right>\phi_j=\hat\a_j\phi_j.
\ee
Thus,  the Koopman modes associated with the data vector of observables $\tilde \f$ are obtained as the left eigenvector $\hat\a_j \phi_j$ of the finite section of the Koopman operator $\tilde U_N$.

Assuming that the approximation of the finite section, the $N\times N$ matrix $\tilde U_N^a$ has distinct eigenvalues $\lambda^a_1,...,\lambda^a_N$, we write the spectral decomposition 
\be
\tilde U_N^a=A\Lambda A^{-1},
\ee
where $\Lambda$ is the diagonal eigenvalue matrix and 
\be
A=\left[ \a_1 \ \a_2\  ... \ \a_N\right ]
\ee
is the column eigenvector matrix.
From
\be
\tilde U_N^a=(F^{\dagger}F)^{-1}F^{\dagger} F'=A\Lambda A^{-1},
\ee
we get that the data can be {\it reconstructed} \index{Data reconstruction} by first observing 
\be
F^{\dagger}F'=F^{\dagger}FA\Lambda A^{-1}.
\label{eq:rec0}
\ee
This represents $N$ equations with $m$ unknowns for each column of $F'$. Assuming $m>N$, it is an underdetermined set of equations  that can have many solutions for columns of $F'$. Then,
\be
F'_p=(FF^{\dagger})^{-1}FF^{\dagger}FA\Lambda A^{-1}=F A\Lambda A^{-1}
\label{eq:rec}
\ee
is the projection of all these solutions on the subspace spanned by the columns of $F$. If $m<N$, (\ref{eq:rec0}) is overdetermined,
and the solution $F'_p$ is the closest - in least squares sense - to $F'$ in the span of the columns of $F$.

Note that $A^{-1}$ is the matrix in which rows are the Koopman modes $\hat\a_k$:
\be A^{-1}=
\left[
\begin{matrix}
\hat \a_1 \\
\vdots \\
\hat \a_N \\
\end{matrix}
\right],
\ee
and thus 
\be \Lambda A^{-1}=
\left[
\begin{matrix}
\lambda_1 \hat \a_1 \\
\vdots \\
\lambda_N \hat \a_N \\
\end{matrix}
\right].
\ee
Using (\ref{eq:eigffs}) we get
\be
FA=
\left[ 
\begin{matrix}
\tilde\f(\x_1)\cdot\a_1 &\cdots&\tilde\f(\x_1)\cdot\a_N\\
\tilde\f(\x_2)\cdot\a_1 &\cdots&\tilde\f(\x_2)\cdot\a_N\\
\vdots&\cdots&\vdots \\
\tilde\f(\x_m)\cdot\a_1 &\cdots&\tilde\f(\x_m)\cdot\a_N\\
\end{matrix}\right]
=\left[ 
\begin{matrix}
\tilde\phi_1(\x_1) &\cdots&\tilde\phi_N(\x_1)\\
\tilde\phi_1(\x_2) &\cdots&\tilde\phi_N(\x_2)\\
\vdots&\cdots&\vdots \\
\tilde\phi_1(\x_m) &\cdots&\tilde\phi_N(\x_m)\\
\end{matrix}\right]
\ee
where $\tilde \phi_j$ is an eigenfunction of the finite section, and $\a_j$'s are the columns of $A$.
Note that $\tilde\phi_k(\x_l)=\tilde\lambda_k^{l-1}\tilde \phi_k(\x_1)$. Using  (\ref{eq:rec}) we get 
\be
F'_p=FA\Lambda A^{-1}=
\left[
\begin{matrix}
\sum_{k=1}^N\tilde \lambda_k \tilde \phi_k(\x_1)\hat \a_k \\
\sum_{k=1}^N\lambda_k^2 \tilde \phi_k(\x_1) \hat \a_k \\
\vdots \\
\sum_{k=1}^N\lambda_k^m  \tilde \phi_k(\x_1)\hat \a_k \\
\end{matrix}
\right].
\ee

\begin{rem} The novelty in this section is the explicit treatment of teh finite section approximation in terms of the dual basis, that enables error estimates in the next subsection. The finite section is also known under the name Galerkin projection \cite{Williamsetal:2015}. The relationship between GLA and finite section methods was studied in \cite{MezicandArbabi:2017}.
\end{rem}
\subsection{Convergence of the Finite Sample Approximation to the Finite Section}
The time averages in (\ref{eq:finsecentries}) converge due to the Birkhoff's Ergodic Theorem \cite{Petersen:1983}.
The rates of convergence depend on the type of asymptotic dynamics that the system is exhibiting:
\begin{ther} Let $T:M\rar M$ be a $C^\infty$ dynamical system with an attractor $A$ and an invariant measure supported on the attractor. Let $U$ be the Koopman operator on $L^2(\mu)$, with a pure point spectrum. Let $f_j,g_k$ be $C^\infty$ for all $j,k$. Then, for almost all $\x$
\be
||\tilde U_{N,m}(\x)-\tilde U_{N}(\x)||_2\leq \frac{c(N)}{m}
\ee
\label{ther:cont}
\end{ther}
\begin{proof}
We suppress the dependence on $\x$ in the notation. The entries  $\tilde u_{kj,m}=\frac{1}{m}\sum_{l=1}^m f_j \circ T(\x_l) \hat g_k^c(\x_l)$ of $U_{N,m}$
(see (\ref{eq:finsam})) converge a.e. w.r.t. $\mu$.  Since $T$ is conjugate to a rotation on an abelian group \cite{neumann1932operatorenmethode}, which is uniformly ergodic,  for sufficiently smooth $T$ and $f_j,\hat g_k$ \cite{mezicDSAMPLE} we have
\be
||\tilde u_{kj,m}-\tilde u_{kj}||_2 \leq \frac{c(f_j,g_k)}{m}
\ee
and the statement follows by setting $c(N)=N^2\max_{j,k} c(f_j,g_k)$.
\end{proof}
\begin{rem}
The smoothness of $T,f_j,\hat g_k$ is required in order for the solution of the homological equation to exist. Only finite smoothness is required \cite{mezicDSAMPLE}, but we have assumed $C^\infty$ for simplicity here.
\end{rem}
The above means that $\tilde U_{N,m}(\x)$ converges to $\tilde U_{N}(\x)$ spectrally:
\begin{cor}
Let $\lambda_m$ be an eigenvalue of tilde $\tilde U_{N,m}(\x)$ with multiplicity $h$. Then, for arbitrary $\eps>0$, for sufficiently large $m>M$, there is an eigenvalue $\lambda$ of $\tilde U_{N}(\x)$ of multiplicity $h$ such that
$
|\lambda_m-\lambda|\leq \eps.
$
\end{cor}
\begin{proof}
This follows from continuity of eigenvalues \cite{texier2017basic} to continuous perturbations (established by theorem \ref{ther:cont}).
\end{proof}
\begin{rem} If $f\circ T^n$ are independent, the convergence estimate above deteriorates to ${\cal{O}}({1/\sqrt{m}})$. 
Presence of continuous spectrum without the strong mixing property can lead to convergence estimates 
${\cal{O}}(1/m^{\alpha})$ with $0<\alpha<1/2$ \cite{kachurovskii1996rate}.
\end{rem}
\begin{rem} Spectral convergence in the infinite-dimensional setting is a more difficult question (see \cite{kordaandmezic:2018} in which only convergence along subsequences was established under certain assumptions). Specifically, we could use the strategy in this section coupled with the boundedness of the Koopman operator to establish $\tilde U$ (the infinite matrix  representation of $U$) converges to $U$ in the operator norm as the number of samples $m$ is increased, which would imply spectral convergence. But the practical question is the convergence in $m$ {\it  and} $N$. To address it further, we start with the formula for error in the finite section.
\end{rem}
\subsection{The Error in the Finite Section}
\index{Finite section! error}
It is of interest to find out how big is the error we are making in the finite section approximations discussed above.
We have the following result.
\begin{prop}
Let  $\tilde\phi=\tilde \e\cdot \tilde \f$ be an eigenfunction of the finite section associated with the eigenvalue $\tilde \lambda$ and eigenvector $\tilde\e$.
Then,
\be
U\tilde\phi-\tilde \lambda \tilde\phi=\tilde\e\cdot(U\tilde\f-P_{\tilde {\mathcal F}}U\tilde\f).
\label{eq:err}
\ee
\end{prop}
\begin{proof}
The first term on the right side of (\ref{eq:err}) follows from the definition of $\tilde \phi$. We then need to show
\be
\tilde \lambda \tilde\phi=\tilde\e\cdot P_{\tilde {\mathcal F}}U\tilde\f.
\ee
But, the left side is just $U_N\tilde \phi$, and since $\tilde \f\in {\tilde {\mathcal F}}$, $U_N\tilde \f=P_{\tilde {\mathcal F}}UP_{\tilde {\mathcal F}}\tilde \f=P_{\tilde {\mathcal F}}U\tilde \f$, which proves the claim.
\end{proof}

\section{Krylov Subspace Methods}
\label{sec:kryl}
A particularly useful feature of dynamical systems theory based on Koopman operator methods is that properties of the
system can be surmised from data. Indeed, in the previous section we found how a finite section of the matrix representation of the Koopman operator can be found from data.
 However, the discussion was based on existence of a basis, that typically might come from taking products on basis elements on $1$-dimensional subspaces - for example Fourier basis on an interval subset of $\bR$. Such constructions lead to an exponential growth in the number of basis elements, and the so-called curse of dimensionality. In this section we study finite section numerical methods that are based on the dynamical evolution of a single or many 
observables - functions on state space - that span the so-called Krylov subspace. The idea is that one might start with a single observable, and due to its evolution span an invariant subspace of the Koopman operator (note the  connection of such methods with the Takens embedding theorem ideas \cite{MezicandBanaszuk:2004,ArbabiandMezic:2017}). Since the number of basis elements is in this case equal to the number of dynamical evolution steps, in any dimension, Krylov subspace-based methods do not suffer from the curse of dimensionality. 
\subsection{Single Observable Krylov Subspace Methods}
Let $T$ be a discrete-time dynamical system on a compact metric space $M$ equipped with a measure $\mu$ on the Borel $\sigma$-algebra.
Let ${\cal F}$ be a Hilbert space of functions on $M$ (for suitable spaces, see \cite{Mezic:2019}).  For a finite-time evolution
of an initial function $f(\x)\in {\cal F}$ under $T$ we get a (Krylov) sequence 
\be 
(f(\x),f\circ T(\x),...,f\circ T^{N}(\x))=(f(\x),Uf(\x),...,f\circ T^{N}(\x)),
\ee
where $U$ is the Koopman
 operator associated with $T$.   Let $f_i=f\circ T^{i-1}(\x)$.
Then clearly $f_{i+1}=Uf_i$, for $i=1,...,N$. If $f_{N+1}$ was in the space spanned by $f_1,...,f_{N}$, and these were linearly independent functions,  we would have
\be f_{N+1}=\sum_{i=1}^{N}c_if_i,\ee
for some constants $c_i,i=1,...N$. In that case, the operator $U$ would have a finite-dimensional approximation $\tilde U_N$
on the $\spn(f_1,...f_{N})$, given by the companion matrix \index{Companion matrix} \index{Matrix! companion}
\be
\tilde U=C=\left(
\begin{matrix}
0 &  0  & \ldots & 0 &c_{1}\\
1 &  0 & \ldots & 0& c_{2}\\
0 &  1 & \ldots & 0& c_{3}\\
\vdots & \vdots & \ddots &  \vdots & \vdots\\
0  &   0       &\ldots & 1 & c_{N} \\
 \end{matrix}\right)
 \label{eq:comp}
\ee
The above is, in the terminology of the previous section, the {\it finite section } representation of $U$. \index{Finite section! companion matrix}
\begin{exa}
Let $V$ be a subspace of $L^2(M)$ spanned by eigenfunctions $e_1,...,e_{N}$ that satisfy
\be
Ue_j=e^{i2\pi\omega_j}e_j
\ee
Let
\be
g=\sum_{j=1}^{N}a_je_j,
\ee
Then 
\bea
U^k g
&=&\sum_{j=1}^{N}a_jU^ke_j \nonumber \\
&=&\sum_{j=1}^{N}a_j e^{i2\pi k\omega_j}e_j, \nonumber \\
\eea
and
\bea
U^N g&=&\sum_{l=1}^{N}d_je_j \nonumber \\
&=&\sum_{j=1}^{N}a_j e^{i2\pi N\omega_j}e_j \nonumber \\
 &=&\sum_{k=1}^{N}c_k\sum_{j=1}^{N}a_j e^{i2\pi k\omega_j}e_j \nonumber \\
 &=&\sum_{j=1}^{N}\left(\sum_{k=1}^{N}c_ke^{i2\pi k\omega_j}\right) a_j  e_j.
 \eea
 Thus, the numbers $c_k,k=1,...,N$ in the companion matrix are determined by $N$ equations with $N$ unknowns
 \be
 \left(\sum_{k=1}^{N}c_ke^{i2\pi k\omega_j}\right) a_j =d_j, \ \ j=1,...,N.
 \ee
 Now let $a_1=1$, $a_j=0, j=2,...,N$. We get $U^N g=d_1 e_1=e^{i2\pi N\omega_1}e_1$, and thus
 \be
c_1=e^{i2\pi N\omega_1}.
 \ee
It is clear that $c_j=0, j=2,...,N$. Note that, if $\omega_1=j/N$ for some integer $j$, we get $c_1=1$, and thus companion matrix becomes the circulant shift matrix
 \index{Circulant shift matrix} \index{Matrix! circulant shift}
 \be
\tilde U=\left(
\begin{matrix}
0 &  0  & \ldots & 0 &1\\
1 &  0 & \ldots & 0& 0\\
0 &  1 & \ldots & 0& 0\\
\vdots & \vdots & \ddots &  \vdots & \vdots\\
0  &   0       &\ldots & 1 & 0 \\
 \end{matrix}\right).
 \label{eq:comp}
\ee
 
\end{exa}
Consider now the case when $U^Nf$ is not in the span of $f_1,...,f_N$. We have the projection formula (\ref{eq:projnonor})
\be
\label{eq:projnonor1}
\tilde u_{kj}=\left<Uf_j, \hat g_k\right>.
\ee
Since 
\be
\left<f_j,\hat g_k\right>=\delta_{k,j},
\ee
for $j=1,...,N-1,k=1,...,N$ we have
\be
\label{eq:projnonor2}
\tilde u_{kj}=\left<Uf_j, \hat g_k\right>=\left<f_{j+1},\hat g_k\right>=\delta_{k,j+1}
\ee
which produces zeros in all columns of row $k$ except in the column ${j-1}$, where we have a $1$. There is no column $j-1$ for row $1$, so 
we get all zeros up to the last column. Now, for the last column we have
\be
\label{eq:projnonor2}
\tilde u_{kN}=\left<Uf_N,\hat g_k\right>=\left<P_{\tilde {\mathcal F}}Uf_N,\hat g_k\right>,
\ee
and thus $c_k$ in the matrix  (\ref{eq:comp}) is the $k$-th coefficient of the orthogonal projection of $Uf_N$ on ${\tilde {\mathcal F}}$ in the basis $\tilde \f$, that here consists  of the Krylov sequence of independent observables \be(f,Uf,...,U^{N-1}f)\equiv(f_1,...,f_N),\ee where we defined $f_1,...,f_N$ by the last relationship.
\subsection{Error in the Companion Matrix Representation}
\index{Companion matrix! representation error}
Let $\tilde \e=(e_1,...,e_N)^T$ be an eigenvector of $\tilde U$ satisfying
\be
\tilde U\tilde \e=\tilde \lambda \tilde \e,
\ee
and $\tilde \f=(f_1,f_2,f_3,...f_{N})$. The action of $U$ on $\tilde \e\cdot\tilde \f$ is given as
\begin{eqnarray}
U\tilde \e\cdot\tilde \f&=&\tilde \e\cdot\tilde \f\circ T=\sum_{i=1}^{N}e_if_{i}\circ T \nonumber \\
              &=& \sum_{i=1}^{N}e_if_{i+1}.
              \label{Koop}
              \end{eqnarray}
Now, we also have
\begin{eqnarray}
\tilde U\tilde \e&=&
\left(
\begin{matrix}
0 &  0  & \ldots & 0 &c_{1}\\
1 &  0 & \ldots & 0& c_{2}\\
0 &  1 & \ldots & 0& c_{3}\\
\vdots & \vdots & \ddots &  \vdots & \vdots \\
0  &   0       &\ldots & 1 & c_{N} \\
 \end{matrix}\right)
 \left(
 \begin{matrix}
  e_1 \\
   e_2 \\
   e_3 \\
  \vdots \\
  e_{N} \\
  \end{matrix}\right) =
  \left(
 \begin{matrix}
  c_1 e_{N}\\
   e_1+c_2e_{N} \\
   e_1+ c_3 e_{N} \\
  \vdots \\
  e_{N-1} +c_{N} e_{N}\\
  \end{matrix}\right) \nonumber \\
  &=&
  \left(
 \begin{matrix}
 0 \\
  e_1 \\
   e_2 \\
  \vdots \\
  e_{N-1} \\
  \end{matrix}\right)+e_{N}
   \left(
 \begin{matrix}
  c_1 \\
   c_2 \\
   c_3 \\
  \vdots \\
  c_{N} \\
  \end{matrix}\right) =\tilde \lambda \tilde \e
\end{eqnarray}

Using this in (\ref{Koop}), and denoting $\tilde\c=(c_1,...,c_N)$, we obtain
     \bea
     U\tilde \e\cdot\tilde \f&=&\lambda \tilde \e\cdot\tilde \f-e_{N}\tilde \c\cdot \tilde \f+e_{N} f\circ T^{N} \nonumber \\
&=&\tilde \lambda \tilde \e\cdot\tilde \f+e_{N}(f\circ T^{N+1}-\tilde \c\cdot \tilde \f).
     \label{pseudo}
  \eea
  This formula also follows directly from (\ref{eq:err}) by observing that $U\tilde\f=(f_2,....,f_N,f_{N+1})$, where $f_{N+1}=f\circ T^N$, the fact that $P(f_2,...,f_N)=(f_2,...,f_N)$, and $P_{\tilde{\mathcal F}}f\circ T^{N}=\tilde\c \cdot\tilde\f$. Thus,
  \be
  \tilde \e\cdot(U\tilde \f-P_{\tilde{\mathcal F}}U\tilde \f)=e_N(f\circ T^N-\tilde\c \cdot \tilde\f).
  \ee
  
We have the following simple consequence:
\begin{lemma} \label{the:find} If $f\circ T^{N}$ is in the $\spn(f_1,...f_{N})$, then $\tilde \phi=\tilde \e\cdot\tilde \f$
     is an eigenfunction of $U$ associated with the eigenvalue $\tilde \lambda$.
     \end{lemma} If the assumption that $f\circ T^{N}$ is in $\spn(f_1,...f_{N})$ is relaxed, and $\tilde \c\cdot \tilde \f$ is the orthogonal projection of $\f\circ T^{N}$
     to  ${\tilde {\mathcal F}}$,  then $\tilde \e\cdot\tilde \f$ is an approximation to the eigenvector of $U$ with an approximate eigenvalue
     $\tilde \lambda$, with the error 
     \be
     e_{N}(f\circ T^N-\tilde\c\cdot \tilde \f)=e_{N}r,
     \label{eq:err}
     \ee
      where $r=f\circ T^N-\tilde\c \cdot \tilde\f$ is called the residual. \index{Residual}
     Note that the equation (\ref{pseudo}) could be written as
     \be
     |U\tilde \e\cdot\tilde \f-\tilde \lambda \tilde \e\cdot\tilde \f|= |e_{N}r|,
     \ee
     which means that $\tilde \e\cdot\tilde \f$ is in the $(\tilde \lambda ,\epsilon)$ -pseudospectrum \index{Pseudospectrum} of $U$ for $\epsilon=|e_{N}r|$ (see \cite{TrefethenandEmbree:2005}).

The calculations above, first presented in \cite{ArbabiandMezic:2017}, allow us to show how the finite section spectrum approximates the spectrum of the Koopman operator when the number of functions $N$ in the Krylov sequence goes to infinity.  The specific sense of approximation here is pseudospectral, and for a class of systems that satisfy a convergence requirement on the Krylov sequence, convergence in pseudospectral sense can be proven:
\begin{lemma} \label{lem:Kryl} Let the Krylov sequence satisfy 
\be
\lim_{N\rar \infty}||f\circ T^N-\tilde\c\cdot \tilde \f||=0
\ee
Then, for and $\epsilon>0$, for large enough $n$, an eigenfunction of the finite section $\tilde \phi$ is in the $\epsilon$-pseudospectrum of $U$.
\end{lemma}
\begin{proof} Without loss of generality, we assume $|\tilde \e|=1$ and thus $|e_N|\leq 1$. From  (\ref{eq:err}), taking $N$ large enough, we get
\be
 |U\tilde \e\cdot\tilde \f-\tilde \lambda \tilde \e\cdot\tilde \f|= |e_{N}(f\circ T^N-\tilde\c\cdot \tilde \f)|\leq |f\circ T^N-\tilde\c\cdot \tilde \f|<\eps,
\ee
which proves the claim.
\end{proof}
\begin{ther} \label{ther:pseudo}Assume that for any function $g$ in the space of observables ${\mathcal F}$ equipped with a norm $||\cdot||$ we have
\be
\label{eq:exp}
g=\sum_{j=1}^\infty c_k \phi_k
\ee
where $\phi_k$ are normalized ($||\phi_k||=1$) eigenfunctions of the Koopman operator associated with eigenvalues $|\lambda_k|\leq 1$, i.e. $U$ has a pure point spectrum in ${\mathcal F}$. Let $\tilde \f$ be the Krylov sequence  generated by $f$ and ${\mathcal F}_f$ the cyclic invariant subspace of $U$ generated by $f$.  Let $P^n_{\tilde {\mathcal F}}$ be the orthogonal projection on the subspace of ${\mathcal F}$ generated by the first $n$ elements of the Krylov sequence.  Then, for any $\eps>0$, there is an $N$ such that $n\geq N$ implies
\be
||U\tilde \phi_k-\tilde\lambda \tilde\phi_k||<\eps.
\ee
\end{ther}

\begin{proof}
Due to lemma \ref{lem:Kryl}, we only need to prove that, under the assumption on the spectrum, 
\be
\lim_{N\rar \infty}||f\circ T^N-\tilde\c\cdot \tilde \f||=0.
\ee
Due to the assumption in equation (\ref{eq:exp}) we have 
\be
\label{eq:ftn}
f=\sum_{k=1}^\infty c_k \phi_k,
\ee
and thus
\be
\label{eq:ftn1}
f^N=f\circ T^N=\sum_{j=1}^\infty c_k \lambda^N_k\phi_k.
\ee
We split the spectrum of $U$ in ${\mathcal F}$ as $\sigma(U)=\sigma(U)|_{S^1}+\sigma(U)|_{D}$, where $D$ is the interior of the unit disk in the complex plane. Then
\be
\label{eq:ftn2}
f\circ T^N=f^N_{S_1}+f^N_{D}=\sum_{\lambda_k \in \sigma(U)|_{S^1}} c_k \lambda^N_k\phi_k+\sum_{\lambda_j \in \sigma(U)|_{D}} c_j \lambda^N_j\phi_j.
\ee
For sufficiently large $N$, for any $\epsilon/2$
\be
\label{eq:ftn3}
|\sum_{\lambda_j \in \sigma(U)|_{D}} c_j \lambda^N_j\phi_j|\leq \eps/2.
\ee
Also, 
\be
\label{eq:ftn4}
\sum_{\lambda_k \in \sigma(U)|_{S^1}} c_k \lambda^N_k\phi_k
\ee
is an almost periodic function, and thus for sufficiently large $M>N$ we have 
\be
\label{eq:ftn5}
|f^M_{S_1}-f^N_{S_1}|\leq\eps/2.
\ee
 Combining (\ref{eq:ftn3}) and (\ref{eq:ftn5}) proves the claim, since $f^M$ is $\eps$-away from an element $f^j$ of the $\spn({f,...,f^{M-1}})$, and $||f\circ T^M-\tilde\c\cdot \tilde \f||$
 is the minimal distance of $f^M$ to the subspace $\spn({f,...,f^{M-1}})$ that contains $f^j$.

\end{proof}
\begin{rem} The above construction only requires the Krylov sequence, and shows that the finite section approximation reveals the pseudospectrum of the Koopman operator. Thus methods relying on Krylov sequences are ``sampling" the high-dimensional space and can approximate the part of the spectrum contained in their invariant subspace irrespective of the dimension of the problem.

The use of Krylov sequences is also of interest because they span the smallest invariant subset that the observable $f$ belongs to:
\begin{ther} Let $f$ be an observable. Then ${\sf span}({f,f\circ T,...,f\circ T^n,...})$ is the smallest forward invariant subspace of $U$ that contains $f$
\end{ther}
\begin{proof} Assume not. Then there is $A\subset {\sf span}({f,f\circ T,...,f\circ T^n,...})$, where $A$ is a proper subset, that contains $f$, meaning that there is $f\circ T^j$ for some integer $j$, that is not in $A$. But then $A$ is not invariant, since it contains $f$ and $U^j f$ is not in $A$.
\end{proof}

\end{rem}
\begin{rem} The assumptions in the theorem  \ref{ther:pseudo} are satisfied by any dynamical system with a quasi-periodic attractor with the space of observables being 
an appropriately constructed Hilbert space \cite{Mezic:2019}. However, they exclude systems with mixed or purely continuous spectrum, as evidenced by the example \ref{exa:mix}.
\end{rem}
\subsection{Krylov Sequences from Data}
If $M$ is not a finite discrete set, numerically, we do not have $\tilde \f$ on the whole state space. Instead, we might be able to {\it sample} the function $f$ on a discrete subset
of points $\X=\{\x_1,....,\x_m\}^T\subset M$. We can think of $f$ as a column vector, and form again the $m\times N$ {\it data matrix} $F$  \index{Data matrix} \index{Matrix! data}
\be
F=\left[ \begin{matrix} f_1(\X) \ f_2(\X)\ ...\ f_{N}(\X)\end{matrix}\right].
\label{datamatrix1}
\ee
and its first iterate 
\bea
F'&=&\left[ \begin{matrix} f_2(\X) \ f_3(\X)\ ...\ f_{N+1}(\X)\end{matrix}\right]\nonumber \\
&=&\left[ \begin{matrix} f_1(T\X) \ f_2(T\X)\ ...\ f_{N}(T\X)\end{matrix}\right] \nonumber \\
&=&\left[ \begin{matrix} f_1(\Y) \ f_2(\Y)\ ...\ f_{N}(\Y),\end{matrix}\right].
\eea
where $\Y=T\X$. We have 
\be
F'=FC,
\ee
or
\be
C=F^+ F'
\ee
as could be surmised from (\ref{eq:finsecmat}),
and the following corollary of the lemma \ref{the:find} holds:
\begin{cor} Let  $f\circ T^N$ be in the $\spn(f_1,...f_N)$, and $\rank F=N$. Let $\tilde\lambda,\tilde \e$ be an eigenvalue and the eigenvector of the companion matrix $\tilde U$. Then, an eigenvalue $\tilde \lambda$ of $\tilde U$ is an eigenvalue of $U$, and $\tilde\f(\X)\cdot \tilde \e$ is a sample of the corresponding eigenfunction of $U$ on $\X$.
\end{cor}
\begin{proof} As soon as we know $N$ samples of the function $f$, the vector $\tilde \c$ in the companion matrix is fixed, and thus the residual is zero.
\end{proof}
When $\x_{k}=T\x_{k-1},k=1,...,m+n-1$, i.e. the sampling points are on a single trajectory, 
the matrix $F$ becomes the Hankel-Takens matrix \index{Hankel-Takens matrix} \index{Matrix! Hankel-Takens}
\be
H=\left[ \begin{matrix} f(\x) \ & \  f(T\x)\ &\ ...\ &\ f(T^{n-1}\x) \\
f(T\x) \ & \  f(T^2\x)\ &\ ...\ &\ f(T^{n}\x) \\ 
\vdots & \cdots & \cdots  & \vdots \\ 
f(T^m\x) \ & \  f(T^{m+1}\x)\ &\ ...\ &\ f(T^{m+n-1}\x) \end{matrix}\right].
\ee
The reason for calling $H$ the Hankel-Takens matrix is that, besides the usual property of Hankel matrices that have
constant skew-diagonal terms - in this case $H_{i,j}=f(T^k\x)$, where $ k=|i|+|j|-2$ - it also satisfies $H_{i,j+1}=H_{i,j}\circ T=H_{i+1,j}$, a property which is related to the Takens embedding \cite{takens:1981,MezicandBanaszuk:2004}.

Let $C$ have distinct eigenvalues. We diagonalize it using 
\be
C=A\Lambda A^{-1},
\ee
The companion matrix is diagonalized by the so-called Vandermonde matrix
\be
A^{-1}={{\begin{bmatrix}1&\lambda _{1}&\lambda _{1}^{2}&\dots &\lambda _{1}^{N-1}\\1&\lambda _{2}&\lambda _{2}^{2}&\dots &\lambda _{2}^{N-1}\\1&\lambda _{3}&\lambda _{3}^{2}&\dots &\lambda _{3}^{N-1}\\\vdots &\vdots &\vdots &\ddots &\vdots \\1&\lambda _{N}&\lambda _{N}^{2}&\dots &\lambda _{N}^{N-1}\end{bmatrix}}.}
\label{eq:van}
\ee
Thus, the Koopman modes of the vector of observables $\tilde \f$ composed of time delays are precisely the columns of the Vandermonde matrix, while the right eigenvectors are the columns of the inverse of the Vandermonde matrix. 
\subsection{ Schmid's  Dynamical Mode Decomposition as a Finite Section Method}

The key numerical issue with the Krylov subspace-based algorithms is the fact that the procedure requires inversion of the Vandermonde matrix (\ref{eq:van}). Since the condition number \index{Condition number} \index{Matrix! condition number} $||A||||A^{-1}||$ (where $||\cdot||$ is the induced matrix norm) of the Vandermonde matrix scales exponentially in its size provided $\lambda_k\neq e^{i\omega}$, for some $k$, even if $|\lambda_k|\approx 1$ \cite{pan:2016}. There are a variety of ways to resolve this issue, and the first one that appeared \cite{Schmid:2010} is the following version of the Koopman operator approximation, based on singular value decomposition. \index{Singular value decomposition}

Let 
\be
F=G\Sigma V^{\dagger}
\ee
be the ``thin" singular value decomposition of the $m\times N$ ``data matrix" $F$, whose columns are samples of functions $f_1,...,f_N$.  The $m\times N$ matrix $G$ and $N \times N$ matrix $V$ are unitary matrices, $V^{\dagger}$ is the conjugate transpose of $V,$ and $\Sigma$ is an $N \times N$ diagonal matrix. Note that 
\be
FV=G\Sigma,
\ee
and thus
\be
F\v_j=\sigma_j\u_j, j=1,...,n
\ee
where $\v_j$ is the $j$-th column of $V$ and $\u_j$ is the $j$-the column of $\u$. Clearly then, 
$\u_j$ are linear combinations of vectors $f_1(\X),f_2(\X),...,f_N(\X)$, and for $m\geq N$ there are $N$ such linear combinations.
We could consider each of these combinations as a sample of a function,
\be
\tilde u_j=\v_j\cdot \tilde \f,
\ee
where $\tilde \f=(f_1,...,f_N)$ is the vector of independent functions. In other words, 
\be
\tilde \u=(u_1,...,u_N)
\ee
 spans ${\cal F}$
and is an orthogonal basis for it. Now $G$ is in fact the data matrix whose columns are $\u_j$'s:
\be
G=[\u_1 \  \u_2 \ ... \ \u_N]=[\u_1(\X)\  \u_2(\X) \ ... \ \u_N(\X)]
\ee
Then, the finite section is 
\bea
\tilde U^S_N&=&G^+ G'=(G^{\dagger}G)^{-1}G^{\dagger}G'\nonumber \\
&=&G^{\dagger} F'V\Sigma^{-1}.
\eea
Now, since 
\be
G^{\dagger}=(FV\Sigma^{-1})^{\dagger}=\Sigma^{-1}V^{\dagger}F^{\dagger},
\ee
\be
G^{\dagger}G=\Sigma^{-1}V^{\dagger}F^{\dagger}FV\Sigma^{-1},
\ee
and thus
\be
(G^{\dagger}G)^{-1}G^{\dagger}=\Sigma V^{\dagger}(F^{\dagger}F)^{-1}V\Sigma \Sigma^{-1}V^{\dagger}F^{\dagger}=\Sigma V^{\dagger}(F^{\dagger}F)^{-1}F^{\dagger},
\ee
we have
\be
\tilde U^S_N=\Sigma V^{\dagger}(F^{\dagger}F)^{-1}F^{\dagger}F'V\Sigma^{-1}=\Sigma V^{\dagger}F^{+}F'V\Sigma^{-1}=\Sigma V^{\dagger}\tilde U^a_NV\Sigma^{-1}.
\ee
Therefore, $\tilde U^S_N$ and $\tilde U^a_N$ are similar matrices that thus have the same spectrum. If $\a_j$ is an eigenvector of $\tilde U^S_N$,
then $V\Sigma^{-1} \a_j$ is an eigenvector of $\tilde U^a_N$, and, according to (\ref{eqn:eigen})
\be
\tilde \phi_j^N= G\a_j 
\ee
is a finite section approximation to an eigenfunction of the Koopman operator.

\section{Weak Eigenfunctions from Data }
\label{sect:weakeig}
In the sections above we  presented finite section approximations of the Koopman operator,  starting from the idea that bounded infinite-dimensional operators are, given a basis,  represented by infinite matrices, and then truncating those. In this section we will present an alternative point of view that  provides additional insights into the relationship between the finite-dimensional approximation and the operator. As a consequence of this approach, we show how the concept of a weak eigenfunction, first discussed in \cite{kordaandmezic:2018}, arises. \index{Weak eigenfunction} \index{Eigenfunction! weak}

We start again with a vector of observables, $\tilde \f=(f_1,...,f_N)$.  Except when we can consider this problem analytically,
we  know the values of observables only on a finite set of points in state space, $\X=\{\x_1,....,\x_m\}$. Assume also that we know the value of $\tilde \f$ at $\Y=\{\y_k\}=\{T(\x_k)\}$. We can think of $f_j(\X)=(f_j(\x_1),...,f_j(\x_m)), j\in \{1,...,N\}$ as a sample of the observable $f_j$ on $\X\subset M$. 

Consider the case $\x_{k+1}=T\x_{k},k=1,...,m-1$. There are many $m\times m$ matrices $A$ such that 
\be
f_j(\Y)^T=Af_j(\X)^T
\ee
One of them is the transpose of the  companion matrix (\ref{eq:comp})
\be
\tilde U^T=\left(
\begin{matrix}
0 &  1  & 0 &\ldots & 0                               \\
0 &  0 & 1& \ldots &  0                                \\
0 &  0 & 0 &\ldots & 0                                 \\
\vdots & \vdots & & \ddots &  \vdots  \\
c_{j1}  &  c_{j2}  & c_{j3} & \ldots  & c_{jm}        \\
 \end{matrix}\right),
 \label{eq:compt}
\ee
but there are many values that $c_{jk},k=1,...m$ can assume, since the only requirement on them is
\be
\sum_{k=1}^{m}c_{jk}f_j(\x_k)=f_j(\y_m)
\ee
and there are $m$ unknowns and $1$ equation that determines them. However, the $c's$ need not depend on $j$, since the operator that maps the vectors $f_j(\X)^T$ to $f_j(\Y)^T$ is not dependent on $j$. Clearly, if there are $m$ observables, then we get
\be
\sum_{k=1}^{m}c_{k}f_j(\x_k)=f_j(\y_m), \ j\in \{1,...,m\},
\ee
and thus we can determine $\c=(c_1,...,c_{m})$ uniquely. 

If the number of observables $N$ is larger than $m$, then $f_{jk}=f_j(\x_k)$ are elements of an 
$N\times m$ matrix\footnote{Note that this data matrix is precisely the transpose of the one we have used before, in (\ref{datamatrix1}).}  $F$ and thus there are not enough components in $\c$ to solve 
\be
F\c=\tilde \f(\y_m)^T.
\ee
This system is overdetermined, so in general does not have a solution. The Dynamic Mode Decomposition method \index{Dynamic Mode Decomposition}
then solves for $\c$ using the following procedure: let $P$ be the orthogonal projection onto span of columns of $F$. Then,
\be
F\c_{MP}=P\tilde \f(\y_m)^T,
\ee
has a solution, provided F has rank $m$:  $P\tilde \f(\y_m)^T$  is an  $N$-dimensional vector in the span of the columns of $F$ and thus can be written as a linear combination of those vectors. In fact, we can write
\be
\c_{MP}=F^+ \tilde \f(\y_m)^T.
\label{MP}
\ee
\index{Moore-Penrose pseudoinverse} 
We now discuss the nature of the approximation of the Koopman operator $U$ by the  companion matrix (\ref{eq:comp})
\be
\tilde U^T=C^T=\left(
\begin{matrix}
0 &  1  & 0 &\ldots & 0                               \\
0 &  0 & 1& \ldots &  0                                \\
0 &  0 & 0 &\ldots & 0                                 \\
\vdots & \vdots & & \ddots &  \vdots  \\
c_{1}  &  c_{2}  & c_{3} & \ldots  & c_{m}        \\
 \end{matrix}\right),
 \label{eq:compt1}
\ee
where $\c=(c_1,...,c_m)=\c_{MP}$ obtained from equation (\ref{MP}). 

Let $\mathcal{S}=\{\x_1,\ldots,\x_m\}$ be an invariant set for $T:M\rar M$, where $M$ is a measure space, with measure $\mu$. Consider the space  ${\mathcal{C}}|_\mathcal{S}$, of continuous functions in $L^2(\mu)$ restricted to $\mathcal{S}$.  This is an $m$-dimensional vector space. The restriction $ U|_\mathcal{S}$ of the Koopman operator to ${\mathcal{C}}|_\mathcal{S}$,  is then a finite-dimensional linear operator that can be represented in a basis by an $m\times m$ matrix. An explicit example is given when $\x_j,j=1,\ldots, m$ represent successive points on a periodic trajectory, and the resulting matrix representation in the standard basis is the  $m\times m$ cyclic permutation matrix
\be
 \Pi=\left(
\begin{matrix}
0 &  1  & 0 &\ldots & 0                               \\
0 &  0 & 1& \ldots &  0                                \\
0 &  0 & 0 &\ldots & 0                                 \\
\vdots & \vdots & & \ddots &  \vdots  \\
1  &  0 & 0  & \ldots  & 0       \\
 \end{matrix}\right),
 \label{eq:compt1}
\ee

If $\mathcal{S}$ is not an invariant set, an $m\times m$ approximation of the reduced Koopman operator can still be provided. Namely, if we know $m$ independent functions' restrictions $(f_j)|_\mathcal{S}$, $j=1,\ldots,m$ in  ${\mathcal{C}}|_\mathcal{S},$ and we also know $f_j(T\x_k)$, $j,k\in\{1,\ldots,m\}$, we can provide a matrix representation of $ U|_\mathcal{S}$. However, while in the case  where $\mathcal{S}$ is an invariant set, the iterate of any function in ${\mathcal{C}}|_\mathcal{S}$ can be obtained in terms of the iterate of $m$ independent functions, for the case when $\mathcal{S}$ is not invariant this is not necessarily so. Namely, the fact that $\mathcal{S}$ is not invariant  means that functions in ${\mathcal{C}}|_\mathcal{S}$ do not necessarily experience linear dynamics under $ U|_\mathcal{S}$. However, one can take $N$ observables $f_j$, $j=1,\ldots,N$, where $N>m$, and approximate the nonlinear dynamics using linear regression on $\tilde \f(\X)\equiv (\f(\x_1),\ldots,\f(\x_m)),$ where $\f(\cdot)=(f_1(\cdot),\ldots,f_N(\cdot))^{\msf T}$ -- i.e by finding an $m\times m$ matrix $C$ that gives the best approximation of the data in the Frobenius norm, 
\begin{equation}
C^T =\underset{B\in\C^{m\times m}}\argmin ||\f(T\x)- \f(\x) {B}||_F \equiv \underset{B\in\C^{m\times m}}\argmin \| (f_j(T\x_k))_{j,k=1,1}^{n,m} -  (f_j(\x_k))_{j,k=1,1}^{n,m} {B}\|_F.
\end{equation}
We have the following:
\begin{ther} Let $T:M\rar M$ be a measure $\mu$-preserving transformation on a metric space $M$, and let $\mathcal{S}_m=\{\x_j\}, j=1,...,m$ be a trajectory such that, when $m\rar\infty$, $\mathcal{S}_m$ becomes dense in a compact invariant set $A\in M$. Then, for any $N$-vector of observables $\f\in {\mathcal{C}}|_{\mathcal{S}_m},N\geq m$, 
we have 
\be
\lim_{m\rar \infty} |U|_{\mathcal{S}_m}\f-C\f|=0
\ee
\end{ther}
\begin{proof} By density of $\mathcal{S}_\infty$, for sufficiently large $M$, $m\geq M$ implies $|\x_m-\x_j|<\eps_M$ for some $\x_j\in {\x_1,...,\x_{m-1}}$. By continuity of observables, 
\be
|U|_{\mathcal{S}_m}\f-C\f|\leq D\eps_M
\ee
for some constant $D$. Taking $M$ sufficiently large makes $\eps_M\rar 0$.
\end{proof}

Consider an m-dimensional eigenvector $\tilde \e=(e_1,...,e_m)$ of $\tilde U^T$, associated with the eigenvalue $\lambda$. 
Since the eigenvector satisfies
\be
\tilde U^T \tilde \e=\lambda \tilde \e
\ee
we have 
\be 
\tilde \e_{k+1}=\lambda \tilde \e_{k},\ k=1,...,m-2.
\ee
Thus $\tilde \e$ can be considered as an eigenfunction on the finite set $\x_1,...,\x_{m-1}$. On the last point of the sample, $\x_m$, 
we have 
\be
\sum_{j=1}^{m}c_j\tilde \e_{j}=\lambda\tilde \e_{m}
\ee
Let us now consider the concept of the  weak eigenfunction, \index{Eigenfunction! weak} or eigendistribution. Let $\nu$ be some prior measure of interest on $M$. Let $\phi$ be a bounded function that satisfies $\phi\circ T=\lambda\phi$. We construct the functional $L$ on $C(M)$ by defining
\be
L(h)= \int_{M}h\phi d\nu.
\ee
Set $UL(h)=\int_{M}h\phi (T\x)d\nu$ and we get
\be
UL(h)=\int_{M}h\phi (T\x)d\nu =\lambda \int_{M}h\phi d\nu=\lambda L(h).
\label{eq:weakeig}
\ee
Clearly, this is satisfied if $\phi$ is a continuous eigenfunction of $U$ at eigenvalue $\lambda$.  But. equation (\ref{eq:weakeig}) is applicable for cases with much less regularity. Namely, if $\mu$ is a measure and 
\be L(f)=\int f(\x)d\mu(\x)
\ee
the associated linear functional,
then we can define the action of $U$ on $L$ by
\be
UL(f)=\int f(\x)d\mu(T\x).
\ee

Consider, for example, a set of points $\x_k,k\in \bN^+$ and assume that for every continuous $h$ there exists the limit 
\be
L(h)=\lim_{K\rar\infty}\frac{1}{K}\sum_{k=1}^K h(\x_k)\tilde \e(\x_k).
\ee
Then, by the Riesz representation theorem \index{Riesz theorem} there is a measure $\mu$, such that
\be
L(h)=\int_M hd\mu.
\ee
\begin{defi}
Let a measure $\mu$ be such that the associated linear functional $L$ satisfies 
\be
UL=\lambda L,
\ee
for some $\lambda \in \bC$. Then $\mu$ is called a  weak eigenfunction of $U$.
\end{defi}
Now we have
\be
UL(h)=\lim_{K\rar\infty}\frac{1}{K}\sum_{k=1}^K h(\x_{k})\tilde \e(\x_{k+1})=\lambda\lim_{K\rar\infty}\frac{1}{K}\sum_{k=1}^K h(\x_{k})\tilde \e(\x_{k})=\lambda L(h),
\ee
proving the following theorem:
\begin{ther} Consider a set of points $\x_k,k\in \bN^+$, on a trajectory of $T$,  and assume that for every continuous $h$ there exists the limit 
\be
L(h)=\lim_{K\rar\infty}\frac{1}{K}\sum_{k=1}^K h(\x_k)\tilde \e(\x_k),
\ee
where \be\tilde \e(\x_k)=\lambda \tilde \e(\x_{k-1}).\ee Then the $\mu$ associated with $L(h)$ by
\be
L(h)=\int_M hd\mu,
\ee
is a weak eigenfunction of $U$ associated with the eigenvalue $\lambda$.
\end{ther}
From the above, it follows that the left eigenvectors of $\tilde U^T$ are approximations of the associated (possibly weak) Koopman modes, as assume $\l$ is such an eigenvector, 
\be
\l_j\tilde U=\lambda_j\l_j.
\ee
then, 
\be
\left<\l,f_j(\X)\right>
\ee
Is the projection of $f_j(\X)$ on the eigenspace spanned by the eigenvector $\e_j$. Moreover, since
\be
\l_j=\lambda\l_{j+1}-c_m \l_m
\ee
the statement can be obtained in the limit $K\rar \infty$ by the so-called Generalized Laplace Analysis (GLA) that we describe in the  section \ref{sect:GLA}.
\begin{rem}The standard interpretation of the Dynamic Mode Decomposition (e.g. on Wikipedia) was in some way a transpose of the one presented here: the observables $f_1^T,...f_m^T$ (interpreted as column vectors) were assumed to be related by a matrix $A:f_{j+1}=Af_j$. Instead, in the nonlinear, Koopman operator interpretation, each row is mapped into its image, and this allows interpretation on the space of observables. This is particularly important in the context of evolution equations, for example fluid flows, where the evolution of observables field - the field of velocity vectors at different spatial points - is not evolving linearly.
\end{rem}

\section{Conclusions}
\label{sect:concl}
In this paper we pursued analysis of two of the major approaches to computation of Koopman operator spectrum: the Generalized Laplace Analysis and the finite section method. We derived approximation results and reinterpreted finite section as a method acting on samples of continuous functions on the state space. The example of a chaotic system with continuous spectrum shows how a failure of the finite section method can occur for that class of systems.
The question of choice of observables is often raised in the context of finite-section approximations such as the EDMD. Specifically, the number of basis functions - e.g. Fourier basis on a box in a $d$-dimensional space - selected as observables can increase exponentially with the dimension $d$. The pseudospectral result proven here shows that choosing time-delayed observables avoids this issue, making time-delayed observations a natural choice. However, it is clear from the example we gave that the finite section method can fail to converge spectrally for systems with continuous spectrum.

One can understand the Krylov subspace approach as sampling by dynamics in the observables space. The weak eigenfunction approach is based on sampling in state space.
Thus, both techniques avoid the curse of dimensionality that methods like EDMD potentially introduce.

{\bf Acknowledgements} I am thankful to Hassan Arbabi and Mathias Wanner for carefully reading the paper and useful comments. This work was supported in part by the DARPA contract HR0011-16-C-0116 and ARO grants W911NF-11-1-0511 and W911NF-14-1-0359.

\label{sect:conc}
\bibliographystyle{unsrt}
\bibliography{MOTDyS,alex,KvN,darpabdd}

\begin{thebibliography}{10}

\bibitem{Koopman:1931}
B.O. Koopman.
\newblock Hamiltonian systems and transformation in {H}ilbert space.
\newblock {\em Proceedings of the National Academy of Sciences of the United
  States of America}, 17(5):315, 1931.

\bibitem{LasotaandMackey:1994}
A.~Lasota and M.~C. Mackey.
\newblock {\em Chaos, Fractals and Noise}.
\newblock Springer-Verlag, New York, 1994.

\bibitem{SinghandManhas:1993}
Raj~Kishor Singh and Jasbir~Singh Manhas.
\newblock {\em Composition operators on function spaces}, volume 179.
\newblock Elsevier, 1993.

\bibitem{MezicandBanaszuk:2004}
I.~Mezi{\'c} and A.~Banaszuk.
\newblock {Comparison of systems with complex behavior}.
\newblock {\em Physica D: Nonlinear Phenomena}, 197(1-2):101--133, 2004.

\bibitem{Mezic:2005}
Igor Mezi{\'c}.
\newblock Spectral properties of dynamical systems, model reduction and
  decompositions.
\newblock {\em Nonlinear Dynamics}, 41(1-3):309--325, 2005.

\bibitem{Mauroyetal:2020}
A.~Mauroy, I.~Mezi\'c, and Y.~Susuki.
\newblock {\em Koopman Operator in Systems and Control}.
\newblock Springer, 2020.

\bibitem{hansen:2010}
Anders~C Hansen.
\newblock Infinite-dimensional numerical linear algebra: theory and
  applications.
\newblock {\em Proceedings of the Royal Society A: Mathematical, Physical and
  Engineering Sciences}, 466(2124):3539--3559, 2010.

\bibitem{Tao}
Terrence Tao.
\newblock {\em The spectral theorem and its converses for unbounded symmetric
  operators}, 2009 (accessed February 3, 2014).

\bibitem{MezicandBanaszuk:2000}
Igor Mezi\'c and Andrzej Banaszuk.
\newblock Comparison of systems with complex behavior: Spectral methods.
\newblock In {\em Proceedings of the 39th IEEE Conference on Decision and
  Control (Cat. No. 00CH37187)}, volume~2, pages 1224--1231. IEEE, 2000.

\bibitem{Kordaetal:2020}
Milan Korda, Mihai Putinar, and Igor Mezi{\'c}.
\newblock Data-driven spectral analysis of the {K}oopman operator.
\newblock {\em Applied and Computational Harmonic Analysis}, 48(2):599--629,
  2020.

\bibitem{Mezic:2013}
Igor Mezi{\'c}.
\newblock Analysis of fluid flows via spectral properties of the koopman
  operator.
\newblock {\em Annual Review of Fluid Mechanics}, 45:357--378, 2013.

\bibitem{MauroyandMezic:2013}
Alexandre Mauroy and Igor Mezi\'c.
\newblock A spectral operator-theoretic framework for global stability.
\newblock In {\em Decision and Control (CDC), 2013 IEEE 52nd Annual Conference
  on}, pages 5234--5239. IEEE, 2013.

\bibitem{MohrandMezic:2014}
Ryan Mohr and Igor Mezi{\'c}.
\newblock Construction of eigenfunctions for scalar-type operators via laplace
  averages with connections to the koopman operator.
\newblock {\em arXiv preprint arXiv:1403.6559}, 2014.

\bibitem{Mezic:2019}
Igor Mezi{\'c}.
\newblock Spectrum of the koopman operator, spectral expansions in functional
  spaces, and state-space geometry.
\newblock {\em Journal of Nonlinear Science}, pages 1--55, 2019.

\bibitem{Schmid:2010}
Peter~J Schmid.
\newblock Dynamic mode decomposition of numerical and experimental data.
\newblock {\em Journal of fluid mechanics}, 656:5--28, 2010.

\bibitem{Rowleyetal:2009}
C.W. Rowley, I.~Mezi{\'c}, S.~Bagheri, P.~Schlatter, and D.S. Henningson.
\newblock Spectral analysis of nonlinear flows.
\newblock {\em Journal of Fluid Mechanics}, 641(1):115--127, 2009.

\bibitem{Williamsetal:2015}
Matthew~O Williams, Ioannis~G Kevrekidis, and Clarence~W Rowley.
\newblock A data-driven approximation of the {K}oopman operator: {E}xtending
  dynamic mode decomposition.
\newblock {\em Journal of Nonlinear Science}, 25(6):1307--1346, 2015.

\bibitem{kordaandmezic:2018}
Milan Korda and Igor Mezi{\'c}.
\newblock On convergence of extended dynamic mode decomposition to the
  {K}oopman operator.
\newblock {\em Journal of Nonlinear Science}, 28(2):687--710, 2018.

\bibitem{SusukiandMezic:2015}
Yoshihiko Susuki and Igor Mezi\c.
\newblock A {P}rony approximation of {K}oopman mode decomposition.
\newblock In {\em 2015 54th IEEE Conference on Decision and Control (CDC)},
  pages 7022--7027. IEEE, 2015.

\bibitem{ArbabiandMezic:2017}
Hassan Arbabi and Igor Mezic.
\newblock Ergodic theory, dynamic mode decomposition, and computation of
  spectral properties of the {K}oopman operator.
\newblock {\em SIAM Journal on Applied Dynamical Systems}, 16(4):2096--2126,
  2017.

\bibitem{MezicandArbabi:2017}
Igor Mezic and Hassan Arbabi.
\newblock On the computation of isostables, isochrons and other spectral
  objects of the {K}oopman operator using the dynamic mode decomposition.
\newblock {\em IEICE Proceedings Series}, 29(A1L-A-1), 2017.

\bibitem{dasandgiannakis:2019}
Suddhasattwa Das and Dimitrios Giannakis.
\newblock Delay-coordinate maps and the spectra of {K}oopman operators.
\newblock {\em Journal of Statistical Physics}, 175(6):1107--1145, 2019.

\bibitem{dundford:1954}
Nelson Dunford.
\newblock Spectral operators.
\newblock {\em Pacific Journal of Mathematics}, 4(3):321--354, 1954.

\bibitem{hunt1998unique}
Fern~Y Hunt.
\newblock Unique ergodicity and the approximation of attractors and their
  invariant measures using {U}lam's method.
\newblock {\em Nonlinearity}, 11(2):307, 1998.

\bibitem{Megginson:2012}
Robert~E Megginson.
\newblock {\em An introduction to Banach space theory}, volume 183.
\newblock Springer Science \& Business Media, 2012.

\bibitem{mauroy2013isostables}
Alexandre Mauroy, Igor Mezi{\'c}, and Jeff Moehlis.
\newblock Isostables, isochrons, and {K}oopman spectrum for the action--angle
  representation of stable fixed point dynamics.
\newblock {\em Physica D: Nonlinear Phenomena}, 261:19--30, 2013.

\bibitem{ArnoldandAvez:1968}
Vladimir~Igorevich Arnold and Andr{\'e} Avez.
\newblock {\em Ergodic Problems of Classical Mechanics, 1968}.
\newblock Benjamin, 1968.

\bibitem{govindarajanetal:2019}
Nithin Govindarajan, Ryan Mohr, Shivkumar Chandrasekaran, and Igor Mezic.
\newblock On the approximation of {k}oopman spectra for measure preserving
  transformations.
\newblock {\em SIAM Journal on Applied Dynamical Systems}, 18(3):1454--1497,
  2019.

\bibitem{Petersen:1983}
K.~Petersen.
\newblock {\em Ergodic Theory}.
\newblock Cambridge University Press, Cambridge, 1995.

\bibitem{klusetal:2016}
Stefan Klus.
\newblock On the numerical approximation of the {P}erron-{F}robenius and
  {K}oopman operator, 2016.

\bibitem{ridge:1973}
William~C Ridge.
\newblock Spectrum of a composition operator.
\newblock {\em Proceedings of the American Mathematical Society},
  37(1):121--127, 1973.

\bibitem{neumann1932operatorenmethode}
J~v Neumann.
\newblock Zur operatorenmethode in der klassischen mechanik.
\newblock {\em Annals of Mathematics}, pages 587--642, 1932.

\bibitem{mezicDSAMPLE}
I.~Mezi\'c.
\newblock Dsample: A deterministic algorithm for sampling with {\cal{o}}(1/n)
  error.
\newblock Preprint.

\bibitem{texier2017basic}
Benjamin Texier.
\newblock Basic matrix perturbation theory.
\newblock {\em Expository note available at www. math. jussieu. fr/\~{}
  texier}, 2017.

\bibitem{kachurovskii1996rate}
Alexander~Grigoryevich Kachurovskii.
\newblock The rate of convergence in ergodic theorems.
\newblock {\em RuMaS}, 51(4):653--703, 1996.

\bibitem{TrefethenandEmbree:2005}
Lloyd~Nicholas Trefethen and Mark Embree.
\newblock {\em Spectra and pseudospectra: the behavior of nonnormal matrices
  and operators}.
\newblock Princeton University Press, 2005.

\bibitem{takens:1981}
Floris Takens.
\newblock Detecting strange attractors in turbulence.
\newblock In {\em Dynamical systems and turbulence, Warwick 1980}, pages
  366--381. Springer, 1981.

\bibitem{pan:2016}
Victor~Y Pan.
\newblock How bad are {V}andermonde matrices?
\newblock {\em SIAM Journal on Matrix Analysis and Applications},
  37(2):676--694, 2016.

\end{thebibliography}

\end{document}